\documentclass[a4paper,11pt]{amsart}

\usepackage[utf8]{inputenc}
\usepackage[english]{babel}
\usepackage[left=1.12in, right=1.12in, top=1.5in, bottom=1.5in]{geometry}

\usepackage{amsmath,amssymb,amsthm,amsfonts}
\usepackage{mathtools}
\usepackage{bbm}

\usepackage{graphicx}
\usepackage{color}
\usepackage{longtable}
\usepackage{multirow}
\usepackage{stmaryrd}
\usepackage{paralist}
\usepackage{appendix}
\usepackage{float}
\allowdisplaybreaks[4]

\newtheorem{theorem}{Theorem}[section]
\newtheorem{lemma}[theorem]{Lemma}

\theoremstyle{definition}

\theoremstyle{remark}
\newtheorem{remark}[theorem]{Remark}

\title[Unitary and Nonunitary Representations of the Heisenberg--Weyl Lie Algebra]{Unitary and Nonunitary Representations of the Heisenberg--Weyl Lie Algebra}

\date{\today}

\begin{document}



\begin{abstract}
We examine unitary and nonunitary representations of the Heisenberg--Weyl Lie algebra
$\mathfrak{hw}_n$, with particular emphasis on tensor products of unitary
representations and on indecomposable nonunitary representations.
In the unitary setting, the irreducible representations with nontrivial central
character are the Schr\"odinger representations, as classified by the
Stone--von Neumann theorem. Although tensor products of these representations are considered in the literature, we give a detailed Lie--algebraic analysis  and construct explicit unitary intertwining operators,
including the case where the central characters sum to zero. 
In the nonunitary setting, we consider a natural realization of $\mathfrak{hw}_n$ as a subalgebra of
the real symplectic Lie algebra $\mathfrak{sp}_{2n+2}(\mathbb R)$ and prove that
every finite--dimensional complex irreducible representation of
$\mathfrak{sp}_{2n+2}(\mathbb{R})$
remains indecomposable upon restriction to $\mathfrak{hw}_n$.
This yields a large natural family of finite--dimensional, nonunitary
indecomposable representations of $\mathfrak{hw}_n$.
\end{abstract}

\author[A. Douglas]{Andrew Douglas$^{1,2,3}$}
\address[]{$^1$Department of Mathematics, New York City College of Technology, City University of New York, Brooklyn, NY, 11201, USA.}
\address[]{$^2$Ph.D. Programs in Mathematics and Physics, Graduate Center, City University of New York, New York, NY, 10016, USA.}

\author[H. de Guise]{Hubert de Guise$^4$}
\address{$^4$Department of Physics, Lakehead University, Thunder Bay, ON,  P7B 5E1, Canada}

\author[J. Repka]{Joe Repka$^3$}
\address{$^3$Department of Mathematics, University of Toronto, Toronto, ON, M5S 2E4, Canada}

\dedicatory{Dedicated to the memory of our friend and collaborator Joe Repka.}

\maketitle

\section{Introduction}

The commutation relation between position $Q$ and momentum $P$ operators (in one dimension) originally obtained by Heisenberg early in his formulation of quantum mechanics by infinite matrices
\begin{align}
     P Q- Q  P: = [ P, Q]=-i\hbar \mathbf{1}\, , \label{eq:horiginal}
\end{align}
where $\hbar=h/2\pi$ is the reduced Planck constant, is central to quantum mechanics and encapsulates some of the unexpected features of the theory such as wave-particle duality.  Quoting Max Born \cite{Born}:

``I shall never forget the thrill which I experienced when I 
succeeded in condensing Heisenberg's ideas on quantum conditions
in the mysterious equation $PQ-QP=\frac{h}{2\pi i}$, which is the centre of the new mechanics and was later found to imply the uncertainty relations.''  

Because $Q$ and $P$ are unbounded operators, the exponential form proposed by Weyl \cite{weyl} for the unitary representation $\pi_\lambda(x,y,t)$ of the Heisenberg-Weyl (or $HW$) \textit{group}, defined using the irreps $\rho$ and $\sigma$ of $\mathbb{R}$, is given (in a notation modified from Folland \cite{folland}) by  
\begin{equation}
\begin{aligned}
\pi_\lambda(x,y,z)&=e^{i(x  Q+y  P+t  Z)}=
    e^{ix  Q}e^{iy   P}e^{i(t+\frac12 xy)  Z}=
    e^{iy   P}e^{ix  Q}e^{i(t-\frac12 xy)  Z}\, ,\label{eq:weyformm} \\
&=e^{i \lambda t +i\frac\lambda2 xy}\sigma(x)\rho(y)=
e^{ i \lambda t-i\frac\lambda2 xy}\rho(y)\sigma(x)\, ,\\
\rho(y)&=e^{i y   P}\, ,\qquad \sigma(x)=e^{i x   Q}
\end{aligned}
\end{equation}
and thereby avoids the domain subtleties associated with the unbounded operators $Q$ and $P$.

We now \textit{define} a Lie algebra version of Eq.~(\ref{eq:horiginal})  more useful for our purposes:
\begin{align}
    [Q, P]&=  Z\, ,\quad 
    [  Q, Z]=0\, ,\quad [  P,  Z]=0\,.
    \label{eq:ourhw}
\end{align}
One would be forgiven to suppose that the central element $Z$ is  proportional to the unit matrix, but taking the trace of 
Eq.~(\ref{eq:horiginal}) immediately shows that this cannot hold for finite--dimensional representations. 

For finitely many degrees of freedom $n$, the $n$ pairs $P_k,\,Q_k$ satisfying the canonical commutation relations
\begin{equation}\label{commutation}
\begin{aligned}
    &[P_k,Q_\ell] = \delta_{k\ell}   Z, \\
    &[P_k, P_\ell] =[Q_k,Q_\ell]=[Z,  P_k]=[
    Z, Q_k]=0,
\end{aligned}
\qquad 1 \le k, \,\ell \le n.
\end{equation}
defining the Heisenberg-Weyl algebra $\mathfrak{hw}_n$. 

 Note that here as well as in Eq.~(\ref{eq:ourhw}), the non-zero commutator is defined without the factor $i$ common in physics. 
 
In this work, some representations of $\mathfrak{hw}_n$ are obtained by differentiating unitary representations of $HW_n$ on the dense invariant Schwartz subspace $\mathcal S(\mathbb R^n)$, while others arise by restriction from finite--dimensional, irreducible representations of $\mathfrak{sp}_{2n+2}(\mathbb R)$. Subtleties related to unbounded operators and domains in infinite--dimensional representations are discussed, for instance, in \cite{DorfmeisterDorfmeister,Garrison-Wong}.

The original relation of Eq.~(\ref{eq:horiginal}) has resonated and still resonates well beyond its original scope.  The Heisenberg--Weyl group $HW_n$ is now ubiquitous in physics and mathematics (see for instance of recent examples in a wide range of areas: \cite{Frahm}\cite{Cushman}\cite{Almheiri}.) 
It plays an important role classical invariant theory,
representation theory of nilpotent Lie groups, reductive algebraic groups, and abelian harmonic analysis, to name but a few \cite{howe, KnappLieGroups}.
In physics, the Heisenberg–Weyl algebra $\mathfrak{hw}_n$ underlies coherent states and phase-space methods (for finite $n$) \cite{Glauber,Sudarshan,Perelomov,Schleich}

These connections are reflected in the rich representation theory of $HW_n$ and $\mathfrak{hw}_n$;  note that in the
mathematics literature, this group is often referred to simply as the
\emph{Heisenberg group}.

In this article, we examine unitary and nonunitary representations of the
Heisenberg--Weyl Lie group $HW_n$ and Lie algebra $\mathfrak{hw}_n$, with a focus on
tensor products in the unitary setting and on indecomposable representations in the
nonunitary setting. 
First, we describe the irreducible unitary representations of $HW_n$
and $\mathfrak{hw}_n$, which are given by the \emph{Schr\"odinger representations},  as classified by the Stone--von Neumann theorem. The unitary representation theory of the Heisenberg--Weyl group is classical and appears widely in the physics literature; see, for example, \cite{ali-englis-rmp, degosson,  stone1932,summers-stone-vn}.

We then examine tensor products of these infinite--dimensional $\mathfrak{hw}_n$-representations and formulate explicit equivalences. Building on classical results
(cf.\ \cite{corwin-greenleaf, folland,  degosson, hall, howe, mackey, thangavelu, VilenkinKlimyk}), we present a Lie--algebraic
treatment of these equivalences and give explicit constructions of the corresponding intertwining operators. While tensor product representations associated with the Heisenberg--Weyl Lie algebra
are discussed in the literature (e.g.~\cite{arai}), a detailed Lie--algebraic
analysis of tensor products of Schr\"odinger representations, including explicit unitary intertwining operators, does not appear to be available. In particular, a systematic representation-theoretic treatment of the case in which the central characters sum to zero 
apparently remains absent from the literature.

Next, we construct a large family of finite--dimensional, nonunitary,
indecomposable representations of $\mathfrak{hw}_n$ via symplectic embedding. 
Specifically, we consider a natural subalgebra of  $\mathfrak{sp}_{2n+2}(\mathbb R)$ that is isomorphic to $\mathfrak{hw}_n$  and show that every finite--dimensional complex
irreducible representation of $\mathfrak{sp}_{2n+2}(\mathbb R)$ remains indecomposable when
restricted to $\mathfrak{hw}_n$.
This yields a large family of finite--dimensional indecomposable
representations of $\mathfrak{hw}_n$.

The article is organized as follows. Section~2 contains a detailed description
of the Heisenberg--Weyl group $HW_n$ and its Lie algebra $\mathfrak{hw}_n$. In Section~3, we define the Schr\"odinger representations of $HW_n$ and $\mathfrak{hw}_n$, and then examine tensor products of Schr\"odinger representations of $\mathfrak{hw}_n$, establishing equivalences through the explicit construction of intertwining operators. In Section~4, we construct a large family of finite--dimensional, indecomposable,
nonunitary representations of $\mathfrak{hw}_n$ through symplectic embedding.  

Finally, we provide in Appendix \ref{sec:appendix} a short comment on and an explicit construction of lowest weight states for each irrep with central characters $\lambda+\mu$, with $\lambda,\mu > 0$, in the decomposition of tensor product of irreps labeled by $\lambda$ and $\mu$.

\section{The Heisenberg--Weyl Group and Lie Algebra}

The \emph{Heisenberg--Weyl group $HW_n$} is a connected, simply connected, noncompact,
nilpotent Lie group. 
$HW_n$ may be realized as a matrix group under matrix multiplication:
\begin{equation}\label{expHW}
HW_n
=
\left\{
\begin{pmatrix}
1 & x & t+\tfrac12 x\cdot y \\
0 & I_n & y \\
0 & 0 & 1
\end{pmatrix}
\;\middle|\;
x\in \mathbb{R}^{1\times n},\; y\in \mathbb{R}^{n\times 1},\; t\in \mathbb{R}
\right\}.
\end{equation}
The center of $HW_n$ is the one--dimensional subgroup
\begin{equation}
\mathcal Z(HW_n)
=
\left\{
\begin{pmatrix}
1 & 0 & t \\
0 & I_n & 0 \\
0 & 0 & 1
\end{pmatrix}
\;\middle|\;
t \in \mathbb{R}
\right\}
\cong \mathbb{R}.
\end{equation}

The \emph{Heisenberg--Weyl Lie algebra $\mathfrak{hw}_n$} is the $(2n+1)$--dimensional
real Lie algebra 
\begin{equation}\label{exphw}
\mathfrak{hw}_n
=
\left\{
\begin{pmatrix}
0 & x & t \\
0 & 0 & y \\
0 & 0 & 0
\end{pmatrix}
\;\middle|\;
x\in \mathbb{R}^{1\times n},\; y\in \mathbb{R}^{n\times 1},\; t\in \mathbb{R}
\right\},
\end{equation}
with Lie bracket given by the commutator. 
If \(e_1,\dots,e_n\) is the standard basis of \(\mathbb R^n\), a basis of $\mathfrak{hw}_n$ is 
\begin{equation}
P_k=
\begin{pmatrix}
0 & e_k^{\top} & 0\\
0 & 0 & 0\\
0 & 0 & 0
\end{pmatrix},
\qquad
Q_k=
\begin{pmatrix}
0 & 0 & 0\\
0 & 0 & e_k\\
0 & 0 & 0
\end{pmatrix},
\qquad
Z=
\begin{pmatrix}
0 & 0 & 1\\
0 & 0 & 0\\
0 & 0 & 0
\end{pmatrix}, \quad 1\le k \le n,
\end{equation}
with commutation relations given in Eq.~(\ref{commutation}).
The element $Z$ spans the center of $\mathfrak{hw}_n$, and
\begin{equation}
[\mathfrak{hw}_n,\mathfrak{hw}_n] = \mathbb{R}Z.
\end{equation}

Since the Heisenberg--Weyl group is connected, simply connected,  and nilpotent, the exponential map from $\mathfrak{hw}_n$ to $HW_n$ is a diffeomorphism (cf. \cite[Proposition 1.62]{folland2025}), in particular it is bijective. Specifically, the representative element of $\mathfrak{hw}_n$ in Eq.~\eqref{exphw} is exponentiated to the representative element of $HW_n$ from Eq.~\eqref{expHW}:
\begin{equation}
\exp\!\left(
\underbrace{
\begin{pmatrix}
0 & x & t \\
0 & 0 & y \\
0 & 0 & 0
\end{pmatrix}
}_{\in\; \mathfrak{hw}_n}
\right)
\; = \;
\underbrace{
\begin{pmatrix}
1 & x & t+\tfrac12 x\cdot y \\
0 & I_n & y \\
0 & 0 & 1
\end{pmatrix}
}_{\in \;HW_n}.
\end{equation}
(Please note: there is no prefactor of $i$ in the exponentiation of the algebra element, as is common usage in physics.  We follow in this the convention used in mathematics, which is especially useful here as we are discussing unitary and non-unitary representations in the same paper.)

\section{Unitary Representations: Schr\"odinger Representations}

The Schr\"odinger representations of $HW_n$ are realized on the Hilbert space
$L^{2}(\mathbb{R}^{n})$,
the space of equivalence classes of measurable functions
$f:\mathbb{R}^{n}\to\mathbb{C}$ satisfying
\begin{equation}
\int_{\mathbb{R}^{n}} |f(u)|^{2}\,du < \infty,
\end{equation}
where functions that agree almost everywhere are identified.
All integrals are taken with respect to Lebesgue measure.
The inner product is given by
\begin{equation}
\langle f,g\rangle
=
\int_{\mathbb{R}^{n}} f(u)\,\overline{g(u)}\,du.
\label{eq:innerproduct}
\end{equation}

  For $\lambda \in \mathbb{R}\setminus\{0\}$, the \emph{Schr\"odinger representation} 
  \begin{equation} \pi_\lambda \colon HW_n \longrightarrow U\bigl(L^2(\mathbb{R}^n)\bigr) 
  \end{equation} 
  is defined by 
  \begin{equation}
  \left(\pi_\lambda\begin{pmatrix}
1 & x & t+\tfrac12 x\cdot y \\
0 & I_n & y \\
0 & 0 & 1
\end{pmatrix}f\right)(u) = e^{ i\left(\lambda t + y\cdot u + \tfrac12 \lambda x\cdot y\right)}\, f(u+\lambda x), 
\end{equation} 
for $f \in L^2(\mathbb{R}^n),\ u \in \mathbb{R}^n$ (cf.\ \cite{folland}).
  In particular, the center acts by 
  \begin{equation} \pi_\lambda\left(\begin{pmatrix}
1 & 0 & t \\
0 & I_n & 0 \\
0 & 0 & 1
\end{pmatrix}\right) = e^{ i\,\lambda t}\,I. \end{equation}

 It is well known that for $\lambda \neq 0$, the representation $\pi_\lambda$ is irreducible, and $\pi_{\lambda}$ and $\pi_{\mu}$ are inequivalent for $\lambda\neq \mu$ (cf.\ \cite{folland}).
Moreover, the following famous theorem establishes that these Schr\"odinger
representations exhaust all infinite--dimensional unitary representations of
$HW_n$. Note that we omit a factor of $2\pi$ from the usual definition of the Schr\"odinger representation given in \cite{folland}, in order to align with standard conventions in physics.

\begin{theorem}[Stone--von Neumann, group version]
For each $\lambda \in \mathbb{R}\setminus\{0\}$, there exists, up to unitary
equivalence, a unique irreducible unitary representation of the Heisenberg--Weyl
group $HW_n$ with central character
\begin{equation}
\begin{pmatrix}
1 & 0 & t \\
0 & I_n & 0 \\
0 & 0 & 1
\end{pmatrix} \longmapsto e^{i\,\lambda t}.
\end{equation}
This representation is unitarily equivalent to the Schr\"odinger
representation $\pi_\lambda$ on $L^2(\mathbb{R}^n)$.
\end{theorem}

\begin{remark}[Vanishing central character]
If $\lambda=0$, the center of $HW_n$ acts trivially, and any unitary representation
of $HW_n$ with this property factors through the abelian quotient
\begin{equation}
HW_n/\mathcal{Z}(HW_n)\cong \mathbb R^{2n}.
\end{equation}
Consequently, every such representation is a unitary representation of an
abelian group and decomposes as a direct integral of one--dimensional characters
\cite[Theorem~7.28]{folland-aha}. In
particular, all irreducible unitary representations with trivial central action
are one--dimensional characters.
\end{remark}

\subsection{Derived Representations of $\mathfrak{hw}_n$}

Let $\mathcal{S}(\mathbb{R}^n)$ denote the \emph{Schwartz space}, consisting of all
smooth functions $f \colon \mathbb{R}^n \to \mathbb{C}$ such that, for every
multi--index $\alpha,\beta \in \mathbb{N}^n$,
\begin{equation}
\sup_{u \in \mathbb{R}^n}
\bigl|\,u^\alpha\, \partial^\beta f(u)\,\bigr| < \infty.
\end{equation}
The space $\mathcal{S}(\mathbb{R}^n)$ is a dense, $\mathfrak{hw}_n$--invariant subspace
of $L^2(\mathbb{R}^n)$ on which the derived Schr\"odinger representation
$d\pi_\lambda$ is well defined.
Differentiating the Schr\"odinger representation $\pi_\lambda$ yields the derived
representation
\begin{equation}
d\pi_\lambda \colon \mathfrak{hw}_n \longrightarrow
\mathrm{End}\bigl(\mathcal{S}(\mathbb{R}^n)\bigr),
\qquad
d\pi_\lambda(X)f=\left. \frac{d}{ds}\pi_\lambda(\exp(sX))f\right|_{s=0}
\end{equation}
It is given explicitly, for $f\in\mathcal S(\mathbb R^n)$, by
\begin{align}
(d\pi_\lambda(Q_k)f)(u) &=  i\, u_k\, f(u), \nonumber \\
(d\pi_\lambda(P_k)f)(u) &= \lambda\, \frac{\partial f}{\partial u_k}(u), \label{eq:action}\\
(d\pi_\lambda(Z)f)(u)   &=  i\,\lambda\, f(u),  \nonumber 
\end{align}
for $1\le k \le n.$ (Note the placement of the factors of $i$ in the above.  For the choice more common in physics, see for instance \cite{Wolf}.) 
These operators satisfy the Heisenberg--Weyl commutation relations Eq.~\eqref{commutation}
\begin{equation}
\begin{aligned}
[d\pi_\lambda(P_k), d\pi_\lambda(Q_\ell)]
&=
\delta_{k \ell}\, d\pi_\lambda(Z),\\
[d\pi_\lambda(P_k),d\pi_\lambda(P_\ell)]&=[d\pi_\lambda(Q_k),d\pi_\lambda(Q_\ell)]=0,\\
[d\pi_\lambda(Z),d\pi_\lambda(P_k)]&=[d\pi_\lambda(Z),d\pi_\lambda(Q_k)]=0,
\end{aligned}
\end{equation}
for $1\le k,\, \ell \le n$.
Let $\mathfrak{hw}_n=\mathfrak p\oplus \mathfrak q\oplus \mathbb R Z$, as vector spaces, where
\begin{equation}
\mathfrak q:=\operatorname{span}_{\mathbb{R}}\{Q_1,\dots,Q_n\}\cong \mathbb R^n,\qquad
\mathfrak p:=\operatorname{span}_{\mathbb{R}}\{P_1,\dots,P_n\}\cong \mathbb R^n.
\end{equation}
Then we may equivalently describe the action of $\mathfrak{hw}_n$ in vector form as
\begin{equation}
\begin{aligned}
(d\pi_\lambda(Q)f)(u)
&=  i\, (Q\cdot u)\, f(u), \\
(d\pi_\lambda(P)f)(u)
&= \lambda\, (P\cdot \nabla) f(u), \\
(d\pi_\lambda(Z)f)(u)
&=  i\,\lambda\, f(u),
\end{aligned}
\end{equation}
for $Q \in \mathfrak q$ and $P \in \mathfrak p$.

\begin{theorem}[Stone--von Neumann, Lie algebra version]
Let $\mathfrak{hw}_n$ be the Heisenberg--Weyl Lie algebra.
Every irreducible unitary representation of $\mathfrak{hw}_n$
arising as the derived representation of a unitary representation of $HW_n$,
with central character
\begin{equation}
d\pi(Z)=  i\,\lambda\, I,
\qquad \lambda \neq 0,
\end{equation}
is unitarily equivalent to the Schr\"odinger representation
$d\pi_\lambda$, realized on the Schwartz space $\mathcal S(\mathbb R^n)$
as a dense invariant subspace of $L^2(\mathbb R^n)$.
\end{theorem}

\subsection{Tensor Products of Schr\"odinger Representations}

Throughout this subsection, let $d\pi_\lambda$, with $\lambda\neq 0$, denote the
derived Schr\"odinger representation of the Heisenberg--Weyl Lie algebra
$\mathfrak{hw}_n$ on the Schwartz space $\mathcal S(\mathbb R^n)$. In the following two theorems, we consider the tensor product $d\pi_\lambda \otimes d\pi_\mu$ for the  two cases $\lambda+\mu\ne0$ and $\lambda+\mu=0$, respectively.

\begin{theorem}\label{thm:nonzerocase}
Let $\lambda,\mu \in \mathbb R\setminus \{0\}$ with $\lambda+\mu \neq 0$.
Then, as representations of the Heisenberg--Weyl Lie algebra $\mathfrak{hw}_n$, we have the unitary equivalence
\begin{equation}
d\pi_\lambda \otimes d\pi_\mu \;\cong\; d\pi_{\lambda+\mu}\otimes \mathbbm 1,
\end{equation}
where $d\pi_{\lambda+\mu}$ acts on $\mathcal S(\mathbb R^n)$ and $\mathbbm 1$ denotes the trivial representation of $\mathfrak{hw}_n$ on an auxiliary copy of
$\mathcal S(\mathbb R^n)$ (i.e., $\mathfrak{hw}_n$ acts by $0$ on that factor). In particular, $d\pi_\lambda \otimes d\pi_\mu$ is equivalent to
$d\pi_{\lambda+\mu}$ with infinite multiplicity.
\end{theorem}
\begin{proof}
As described above,
\begin{equation}
\begin{aligned}
(d\pi_\lambda(Q)f)(u)
&=  i\, (Q\cdot u)\, f(u), \\
(d\pi_\lambda(P)f)(u)
&= \lambda\, (P\cdot \nabla) f(u), \\
(d\pi_\lambda(Z)f)(u)
&=  i\,\lambda\, f(u).
\end{aligned}
\end{equation}
Then, the tensor product representation on  
\begin{equation}
\mathcal S(\mathbb R^n)\otimes\mathcal S(\mathbb R^n)
\;\cong\;
\mathcal S(\mathbb R^{2n}),
\end{equation}
with $f, g\in \mathcal{S}(\mathbb R^n)$ and $(u,v)\in \mathbb R^n \times \mathbb R^n$ satisfies
\begin{equation}
\begin{aligned}
\bigl((d\pi_\lambda\otimes d\pi_\mu)(Q) (f \otimes g)\bigr)(u,v)
&= i\,(Q \cdot (u+v))\, (f \otimes g)(u,v),\\
\bigl((d\pi_\lambda\otimes d\pi_\mu)(P)(f \otimes g)\bigr)(u,v)
&=\left( \lambda \, P\cdot \nabla_u + \mu\, P\cdot \nabla_v \right)(f \otimes g)(u,v),
\end{aligned}
\end{equation}
and
\begin{equation}
\bigl((d\pi_\lambda\otimes d\pi_\mu)(Z)(f \otimes g)\bigr)(u,v)
= i(\lambda+\mu)\, (f \otimes g)(u,v).
\end{equation}
Define a unitary operator
\begin{equation}
\mathcal U_{\lambda,\mu} : L^2(\mathbb R^{n} \times \mathbb R^{n}) \longrightarrow L^2(\mathbb  R^{n} \times \mathbb R^{n})
\end{equation}
by
\begin{equation}
\begin{aligned}
(\mathcal U_{\lambda,\mu} (f \otimes g))(u,v)
 &=
|\lambda+\mu|^{-n/2}\,
f \left(\frac{\lambda\, u+v}{\lambda+\mu}\,\right) \, g\left(\frac{\mu\, u-v}{\lambda+\mu}\right)\\
 &=
|\lambda+\mu|^{-n/2}\,
(f\otimes g) \!\left(\frac{\lambda\, u+v}{\lambda+\mu}\,, \, \frac{\mu\, u-v}{\lambda+\mu}\right).
\end{aligned}
\end{equation}
Since the linear change of variables
\begin{equation}
(u,v)\longmapsto\left(\frac{\lambda u+v}{\lambda+\mu},\frac{\mu u-v}{\lambda+\mu}\right)
\end{equation}
has Jacobian determinant of absolute value $|\lambda+\mu|^{-n}$, the normalization
factor $|\lambda+\mu|^{-n/2}$ ensures that $\mathcal U_{\lambda,\mu}$ is unitary on
$L^2(\mathbb R^{2n})$.

Since $(u,v)\mapsto (U,V)$ is an invertible linear change of variables, composition
with this map preserves the Schwartz space. Multiplication by the constant
$|\lambda+\mu|^{-n/2}$ clearly does as well. Hence $\mathcal U_{\lambda,\mu}$
preserves $\mathcal S(\mathbb R^{n})\otimes\mathcal S(\mathbb R^{n})$.
Moreover, it intertwines
$d\pi_\lambda\otimes d\pi_\mu$ with $d\pi_{\lambda+\mu}\otimes\mathbbm 1$.
 That is,
\begin{equation}
\mathcal U_{\lambda,\mu}\circ(d\pi_\lambda\otimes d\pi_\mu)(X)
=
(d\pi_{\lambda+\mu}\otimes \mathbbm{1})(X)\circ \mathcal U_{\lambda,\mu}
\quad\text{for all}\quad X\in\mathfrak{hw}_n,
\end{equation}
and we verify this for each of $Q\in \mathfrak{q}, \;P\in \mathfrak p$ and $Z$. 
We first define 
\begin{equation}
U\coloneqq  \frac{\lambda \, u +v}{(\lambda+\mu)}\quad \text{and}
\quad V\coloneqq \frac{\mu \, u -v}{(\lambda+\mu)},
\end{equation}
so that
\begin{equation}
\begin{aligned}
    (\mathcal U_{\lambda,\mu} (f \otimes g))(u,v)
& =
|\lambda+\mu|^{-n/2}\, f(U) \,g(V)\\
& =|\lambda+\mu|^{-n/2}\, (f\otimes g) (U, V).
\end{aligned}
\end{equation}
Note that 
\begin{equation}(d \pi_{\lambda+\mu}\otimes\mathbbm 1) (X) =
d \pi_{\lambda+\mu}(X) \otimes I,\; \text{for all} \; X\in \mathfrak{hw}_n.
\end{equation}

Now, let $Q\in \mathfrak{q}$, then
\begin{equation}
\begin{aligned}
& \, \mathcal{U}_{\lambda, \mu}((d\pi_\lambda \otimes d \pi_\mu)(Q)(f\otimes g))(u,v) \\
=&\,|\lambda+\mu|^{-n/2}\, (d\pi_\lambda \otimes d \pi_\mu)(Q)(f\otimes g) (U, V)\\
=&\,  i \,|\lambda+\mu|^{-n/2}\, \,(Q \cdot ( U+V)) (f\otimes g)(U, V) \\
=&\,  i \, |\lambda+\mu|^{-n/2}  \,\left(Q \cdot  u \right) (f\otimes g)(U, V) \\
= &\,  i \, (Q\cdot u) ( \mathcal{U}_{\lambda, \mu} (f\otimes g))(u,v)\\
= &\,(d \pi_{\lambda+\mu}(Q) \otimes I)( \mathcal{U}_{\lambda, \mu} (f\otimes g))(u,v)\\
= &\,(d \pi_{\lambda+\mu} \otimes \mathbbm{1} )(Q)( \mathcal{U}_{\lambda, \mu} (f\otimes g))(u,v)\, ,
\end{aligned}
\end{equation}
where $U+V=u$ has been used.
Next, let  $P\in\mathfrak p$, then
\begin{equation}
\begin{aligned}
&(d \pi_{\lambda+\mu} \otimes I )(P)\, (\mathcal{U}_{\lambda, \mu} (f\otimes g))(u,v) \\ 
=\, &(d \pi_{\lambda+\mu}(P) \otimes I )\, (\mathcal{U}_{\lambda, \mu} (f\otimes g))(u,v) \\ 
=\, & (\lambda+\mu)\, (P \cdot \nabla_u)\bigl(\mathcal{U}_{\lambda, \mu} (f\otimes g)\bigr)(u,v) \\
=\,&(\lambda+\mu)\,|\lambda+\mu|^{-n/2}\,(P \cdot \nabla_u)(f\otimes g)(U, V) \\
=\,&(\lambda+\mu)\,|\lambda+\mu|^{-n/2}
\left(
((P\cdot\nabla f) \otimes g) \frac{\partial U }{\partial u}
+
(f \otimes (P\cdot\nabla g))\frac{\partial V}{\partial u}
\right) (U, V) \\
=\,&(\lambda+\mu)\,|\lambda+\mu|^{-n/2}
\left(
\frac{\lambda}{\lambda+\mu}\,((P\cdot\nabla f) \otimes g)
+
\frac{\mu}{\lambda+\mu}\,(f \otimes (P\cdot\nabla g))
\right) (U, V) \\
=\,&|\lambda+\mu|^{-n/2}
\left(
\lambda \,((P\cdot\nabla f) \otimes g)
+
\mu \,(f \otimes (P\cdot\nabla g))
\right) (U, V) \\
=\,&|\lambda+\mu|^{-n/2}\,(d\pi_\lambda(P) \otimes I + I \otimes d \pi_\mu(P) )(f\otimes g)(U, V)\\
=\,&|\lambda+\mu|^{-n/2}\,(d\pi_\lambda \otimes d \pi_\mu)(P)\,(f\otimes g)(U, V)\\
=\,& \mathcal{U}_{\lambda, \mu}\bigl((d\pi_\lambda \otimes d \pi_\mu)(P) (f\otimes g)\bigr)(u, v).
\end{aligned}
\end{equation}

Finally, for the central element $Z$ we have
\begin{equation}
\begin{aligned}
(d \pi_{\lambda+\mu} \otimes \mathbbm{1} )(Z)( \mathcal{U}_{\lambda, \mu} (f\otimes g))(u,v) &= \,  i (\lambda+\mu) \, \mathcal{U}_{\lambda, \mu} (f\otimes g)(u,v)\\
&= \,  \mathcal{U}_{\lambda, \mu} ( i (\lambda+\mu)(f\otimes g))(u,v)\\
&= \, \mathcal{U}_{\lambda, \mu}  (d\pi_\lambda (Z)\otimes I + I \otimes d \pi_{\mu}(Z))(f\otimes g)(u,v)\\
&= \, \mathcal{U}_{\lambda, \mu} ( (d\pi_\lambda \otimes d \pi_\mu)(Z)(f\otimes g))(u,v).
\end{aligned}
\end{equation}
Conjugation by $\mathcal U_{\lambda,\mu}$ gives
\begin{equation}
\begin{aligned}
\mathcal U_{\lambda,\mu}\,(d\pi_\lambda\otimes d\pi_\mu)(Q)\,\mathcal U_{\lambda,\mu}^{-1}
&= (d\pi_{\lambda+\mu}(Q)\otimes I)
= (d\pi_{\lambda+\mu}\otimes\mathbbm 1)(Q),
\qquad Q\in\mathfrak q,\\[4pt]
\mathcal U_{\lambda,\mu}\,(d\pi_\lambda\otimes d\pi_\mu)(P)\,\mathcal U_{\lambda,\mu}^{-1}
&= (d\pi_{\lambda+\mu}(P)\otimes I)
= (d\pi_{\lambda+\mu}\otimes\mathbbm 1)(P),
\qquad P\in\mathfrak p,\\[4pt]
\mathcal U_{\lambda,\mu}\,(d\pi_\lambda\otimes d\pi_\mu)(Z)\,\mathcal U_{\lambda,\mu}^{-1}
&= (d\pi_{\lambda+\mu}(Z)\otimes I)
= (d\pi_{\lambda+\mu}\otimes\mathbbm 1)(Z).
\end{aligned}
\end{equation}
Hence,
\begin{equation}
d\pi_\lambda \otimes d\pi_\mu \;\cong\; d\pi_{\lambda+\mu}\otimes \mathbbm 1,
\end{equation}
that is, $d\pi_\lambda \otimes d\pi_\mu$ is equivalent to
$d\pi_{\lambda+\mu}$ with infinite multiplicity.
\end{proof}

\begin{remark}[Non-irreducibility of the tensor product]
Although $d\pi_{\lambda+\mu}$ is irreducible, the tensor product
\begin{equation}
d\pi_\lambda\otimes d\pi_\mu \cong d\pi_{\lambda+\mu}\otimes \mathbbm 1
\end{equation}
is \emph{not} irreducible as a representation of $\mathfrak{hw}_n$, since
$\mathbbm 1$ is infinite--dimensional and any subspace
$W\subset \mathcal S(\mathbb R^n)$ yields a proper $\mathfrak{hw}_n$--submodule
$\mathcal S(\mathbb R^n)\otimes W$.
\end{remark}

\begin{theorem}\label{thm:zerocase}
Let $\lambda \in \mathbb R\setminus\{0\}$.  
Then the tensor product representation
$d\pi_\lambda \otimes d\pi_{-\lambda}$ of the Heisenberg--Weyl Lie algebra
$\mathfrak{hw}_n$ has trivial central character and hence factors through the
abelian quotient
\begin{equation}
\mathfrak{hw}_n / \mathbb R Z \;\cong\; \mathbb R^{2n}.
\end{equation}
More precisely, there exists a unitary intertwining operator identifying
$d\pi_\lambda \otimes d\pi_{-\lambda}$ with a tensor product representation
\begin{equation}
d\pi_\lambda \otimes d\pi_{-\lambda}
\;\cong\;
\sigma_U \otimes \tau_V
\quad\text{on}\quad
L^2(\mathbb R^n_U)\otimes L^2(\mathbb R^n_V),
\end{equation}
where $\sigma_U$ and $\tau_V$ are representations with commuting images of the
abelian Lie algebra $\mathfrak{hw}_n/\mathbb R Z \cong \mathbb R^{2n}$, realized in
the proof by multiplication and differentiation operators, respectively.
\end{theorem}

\begin{proof}
Since
\begin{equation}
\begin{aligned}
  & ((d\pi_\lambda\otimes d\pi_{-\lambda})(Z)(f\otimes g))(u,v) \\
  =\;& ((d\pi_\lambda(Z) \otimes I + I\otimes d\pi_{-\lambda}(Z))(f\otimes g))(u,v)\\
  =\;& ( i\,\lambda -  i\,\lambda)\, f(u)g(v)\\
  =\;& 0,
\end{aligned}
\end{equation}
the center $\mathfrak z (\mathfrak{hw}_n)=\mathbb R Z$ acts trivially, and the representation
factors through the abelian quotient
$\mathfrak{hw}_n/\mathbb R Z \cong \mathbb R^{2n}$.

Define a unitary operator
\begin{equation}
\mathcal U:L^2(\mathbb R^{n} \times \mathbb R^{n})
\longrightarrow
L^2(\mathbb R^{n} \times \mathbb R^{n})
\end{equation}
by
\begin{equation}
(\mathcal U (f\otimes g))(u,v)
=
(f\otimes g)\!\left(\tfrac{u+v}{\sqrt2},\,\tfrac{u-v}{\sqrt2}\right).
\end{equation}
The change of variables
\begin{equation}
(u,v)\mapsto\left(\tfrac{u+v}{\sqrt2},\tfrac{u-v}{\sqrt2}\right)
\end{equation}
is orthogonal, hence has Jacobian of absolute value $1$, so $\mathcal U$ is unitary on $L^2(\mathbb R^{2n})$.
Moreover, $\mathcal U$ preserves $\mathcal S(\mathbb R^n)\otimes\mathcal S(\mathbb R^n)$.
Define
\begin{equation}\label{UVdef}
U\coloneqq \frac{u+v}{\sqrt2},
\qquad
V\coloneqq \frac{u-v}{\sqrt2},
\end{equation}
so that
\begin{equation}
(\mathcal U (f\otimes g))(u,v)
=
(f\otimes g)(U,V).
\end{equation}

Let $Q\in \mathfrak q$, and define $(M_Q f)(U)\coloneqq(Q\cdot U)f(U)$. Then
\begin{equation}
\begin{aligned}
(d\pi_\lambda \otimes d \pi_{-\lambda})(Q)
(\mathcal{U} (f\otimes g))(u,v)
&=
 i (Q\cdot (u+v)) (f\otimes g)(U,V) \\
&=
\sqrt2\, i (Q\cdot U) (f\otimes g)(U,V) \\
&=
\sqrt2\, i ((Q\cdot U)f(U)) g(V) \\
&=
\sqrt2\, i ((M_Qf)(U)) g(V) \\
&=
\sqrt2\, i ((M_Qf) \otimes g) (U, V) \\
&=
\mathcal{U}
\bigl(\sqrt2\, i\, (M_Q f \otimes g)\bigr)(u,v).
\end{aligned}
\end{equation}


Next, for $P\in\mathfrak p$, 
\begin{equation}
\begin{aligned}
(d\pi_\lambda \otimes d\pi_{-\lambda})(P)
(\mathcal U(f\otimes g))(u,v)
&=
\bigl(\lambda\, P\cdot \nabla_u - \lambda\, P\cdot \nabla_v\bigr)
(f \otimes g)(U, V) \\
&= \sqrt2\,\lambda\,
(P\cdot \nabla_V)\bigl((f \otimes g)(U, V)\bigr) \\
&= \sqrt2\,\lambda\, f(U)\,(P\cdot\nabla_V g)(V) \\
&= \sqrt2\,\lambda\,
(f\otimes (P\cdot\nabla g))(U,V) \\
&=
\mathcal U
\bigl(\sqrt2\,\lambda\, (f \otimes (P\cdot \nabla g))\bigr)(u,v).
\end{aligned}
\end{equation}

Conjugation by $\mathcal U$ yields
\begin{equation}
\begin{aligned}
\mathcal U^{-1}\,(d\pi_\lambda\otimes d\pi_{-\lambda})(Q)\,\mathcal U
&= (\sqrt2\, i\, M_Q)\otimes I, \\
\mathcal U^{-1}\,(d\pi_\lambda\otimes d\pi_{-\lambda})(P)\,\mathcal U
&= I\otimes(\sqrt2\,\lambda\, P\cdot\nabla),
\end{aligned}
\end{equation}
and the central element $Z$ acts trivially.
Thus
\begin{equation}
d\pi_\lambda \otimes d\pi_{-\lambda}
\;\cong\;
\sigma_U \otimes \tau_V,
\end{equation}
where $\sigma_U$ and $\tau_V$ are representations of the abelian Lie algebra
$\mathfrak{hw}_n/\mathbb R Z$ defined by
\begin{equation}
\sigma_U(Q)=\sqrt2\, i\, M_Q,
\qquad
\sigma_U(P)=0,
\qquad
\tau_V(Q)=0,
\qquad
\tau_V(P)=\sqrt2\,\lambda\, P\cdot\nabla,
\end{equation}
with $Z$ acting trivially in both representations.
\end{proof}

\begin{remark}[Outside the scope of the Stone--von Neumann framework]
Conceptually, after passing to the abelian quotient
$\mathfrak{hw}_n/\mathbb R Z \cong \mathbb R^{2n}$, the tensor product
exhibits the decomposition
\begin{equation}
d\pi_\lambda \otimes d\pi_{-\lambda}
\;\cong\;
(\text{multiplication}) \otimes (\text{differentiation}),
\end{equation}
reflecting the fact that the Stone--von Neumann classification no longer
applies when the central character vanishes.
At the group level, the corresponding unitary representation therefore
factors through the abelian group $\mathbb R^{2n}$ and, by the spectral
theorem for unitary representations of locally compact abelian groups,
decomposes schematically as
\begin{equation}
\pi_\lambda \otimes \pi_{-\lambda}
\;\cong\;
\int^\oplus \chi \, d\mu(\chi),
\end{equation}
that is,
as a direct integral of one--dimensional unitary characters
\cite[Theorem~7.28]{folland-aha}.
\end{remark}

\section{Nonunitary Representations via Symplectic Embedding}

In this section, we construct a large family of finite--dimensional, nonunitary complex
indecomposable representations of $\mathfrak{hw}_n$.
Specifically, we identify a natural subalgebra of $\mathfrak{sp}_{2n+2}(\mathbb R)$
that is isomorphic to $\mathfrak{hw}_n$, and show that every finite-dimensional
complex irreducible representation of $\mathfrak{sp}_{2n+2}(\mathbb R)$ remains
indecomposable when restricted to $\mathfrak{hw}_n$.

We first review relevant background on closed subsets of root systems and regular
subalgebras. In addition, we examine the real and complex symplectic algebras together with 
corresponding natural subalgebras isomorphic to  the Heisenberg--Weyl Lie algebra and its
complexification, respectively.

\subsection{Closed Subsets of Root Systems and Regular Subalgebras}

This subsection largely follows the treatment from \cite{DouglasRepka2025}.
Let $\mathfrak{g}^\mathbb{C}$ denote a complex semisimple Lie algebra, and let
$\mathfrak{h}^\mathbb{C}$ be a fixed Cartan subalgebra of $\mathfrak{g}^\mathbb{C}$. Denote by $\Phi$ the corresponding root system of
$(\mathfrak g^{\mathbb C},\mathfrak h^{\mathbb C})$.
The set of positive roots is denoted $\Phi^+$. 
For $\alpha\in\Phi$, we write $\mathfrak g_\alpha$ for the root space
$\mathfrak{g}_{\alpha}^{\mathbb{C}} \subset \mathfrak g^{\mathbb C}$.

A subset $T\subset\Phi$ is said to be \emph{closed} if for any
$\alpha,\; \beta\in T$, the condition $\alpha+\beta\in\Phi$ implies
$\alpha+\beta\in T$.
Any closed subset $T$ decomposes as a disjoint union
\begin{equation}
T = T^r \sqcup T^u,
\end{equation}
where
\begin{equation}
T^r = \{\alpha\in T \mid -\alpha\in T\},
\qquad
T^u = \{\alpha\in T \mid -\alpha\notin T\},
\end{equation}
are called the \emph{symmetric} and \emph{special} components of $T$,
respectively.

A \emph{regular subalgebra} of a complex semisimple Lie algebra is a subalgebra
that is normalized by a Cartan subalgebra.
In particular, let $T\subset\Phi$ be a closed subset, and let
$\mathfrak t\subset\mathfrak h^{\mathbb C}$ be a subspace containing
$[\mathfrak g_\alpha,\mathfrak g_{-\alpha}]$ for each $\alpha\in T^r$.
Then
\begin{equation}\label{reggg}
\mathfrak s_{T,\mathfrak t}
=
\mathfrak t \;\oplus \bigoplus_{\alpha\in T}\mathfrak g_\alpha
\end{equation}
is a regular subalgebra of $\mathfrak g^{\mathbb C}$.
For convenience, if $\mathfrak t=0$, we write $\mathfrak s_T$ in place of
$\mathfrak s_{T,0}$.

Let $S\subseteq\Phi$.
The \emph{closure} of $S$, denoted $[S]$, is the smallest closed subset of
$\Phi$ containing $S$.
The following result is well known and routine to verify.

\begin{lemma}\label{semireg}
Let $T$ be a closed subset of $\Phi$.
Then $[T\cup(-T)]$ is a symmetric closed subset of $\Phi$ containing $T$.
\end{lemma}

We end this subsection with a theorem giving necessary and sufficient conditions for finite--dimensional complex irreducible representation of $\mathfrak{g}^\mathbb{C}$ to remain indecomposbale when restricted to a regular subalgebra $\mathfrak{s}_{T,\mathfrak{t}}$.
It is one of the results of Douglas and Repka in \cite{DouglasRepka2025}, and generalizes earlier work of Panyushev \cite{panyushev}.

\begin{theorem}\cite[Corollary~3.10]{DouglasRepka2025}\label{lwide}
Let $\mathfrak g^{\mathbb C}$ be a complex semisimple Lie algebra with Cartan
subalgebra $\mathfrak h^{\mathbb C}$ and root system $\Phi$.
Let $T\subset\Phi$ be a closed subset, and let
$\mathfrak t\subset\mathfrak h^{\mathbb C}$ be a subalgebra containing
$[\mathfrak g_\alpha,\mathfrak g_{-\alpha}]$ for each $\alpha\in T^r$.
Set
\begin{equation}
\mathfrak s_{T,\mathfrak t}
=
\mathfrak t \;\oplus \bigoplus_{\alpha\in T}\mathfrak g_\alpha.
\end{equation}
Then, every finite--dimensional complex irreducible representation of
$\mathfrak g^{\mathbb C}$ remains indecomposable when restricted to
$\mathfrak s_{T,\mathfrak t}$ if and only if
\begin{equation}
[T\cup (-T)] = \Phi.
\end{equation}
\end{theorem}

\subsection{The Real and Complex Symplectic Algebras}

The
\emph{complex symplectic algebra}  is given by
\begin{equation}
\begin{aligned}
\mathfrak{sp}_{2n+2}(\mathbb{C})
&=\{X\in M_{2n+2}(\mathbb{C}) \mid X^{\top}J+JX=0\},
\end{aligned}
\end{equation}
where
\begin{equation}
J=\begin{pmatrix}
0 & I \\
-I & 0
\end{pmatrix}.
\end{equation}
We may write this simple Lie algebra in block form as follows:
\begin{equation}\label{block}
\mathfrak{sp}_{2n+2}(\mathbb{C})
=
\left\{
\begin{pmatrix}
A & B\\
C & -A^{\top}
\end{pmatrix}
\;\middle|\;
A\in M_{n+1}(\mathbb{C}),\;
B^{\top}=B,\;
C^{\top}=C
\right\}.
\end{equation}

Let $\mathfrak h \subset \mathfrak{sp}_{2n+2}(\mathbb C)$ be the diagonal Cartan subalgebra
\begin{equation}
\mathfrak h
=
\left\{
\begin{pmatrix}
T & 0 \\
0 & -T
\end{pmatrix}
\;\middle|\;
T = \mathrm{diag}(t_1,\dots,t_{n+1}), \;\; t_1,\dots t_{n+1} \in \mathbb{C}
\right\}.
\end{equation}
For $1 \le k \le n+1$, define $\varepsilon_k \in \mathfrak h^{*}$ by
\begin{equation}
\varepsilon_k\left(\begin{pmatrix}
T & 0 \\
0 & -T
\end{pmatrix}\right) = t_k.
\end{equation}

The corresponding root system of $\mathfrak{sp}_{2n+2}(\mathbb C)$, which is of type $C_{n+1}$, is then
\begin{equation}
\Phi
=
\{\pm(\varepsilon_k-\varepsilon_\ell),\ \pm(\varepsilon_k+\varepsilon_\ell)
\; |\; 1\le k<\ell\le n+1 \}
\;\cup\;
 \{\pm 2\varepsilon_k \; |\; 1\le k\le n+1 \}.
\end{equation}
A standard choice of simple roots is
\begin{equation}
\alpha_k=\varepsilon_k-\varepsilon_{k+1}\quad (1\le k\le n),
\qquad
\alpha_{n+1}=2\varepsilon_{n+1}.
\end{equation}

With respect to this choice, the corresponding set of positive roots is
\begin{equation}
\Phi^+
=
\{\varepsilon_k-\varepsilon_\ell,\ \varepsilon_k+\varepsilon_\ell
\; |\; 1\le k<\ell\le n+1\}
\;\cup\;
\{2\varepsilon_k \; |\; 1\le k\le n+1 \}.
\end{equation}

We now describe explicit root vectors for each root. Let $E_{k \ell}$ denote the $(n+1) \times (n+1)$ matrix unit with a $1$ in the $(k,\ell)$--entry and zeros
elsewhere.
\vspace{.7em}

\noindent \underline{Roots $\varepsilon_k - \varepsilon_\ell$ ($k \neq \ell$)}: A root vector for the root $\varepsilon_k - \varepsilon_\ell$ is 
\begin{equation}
X_{\varepsilon_k - \varepsilon_\ell}
\coloneqq
\begin{pmatrix}
E_{k \ell} & 0 \\
0 & -E_{\ell k}
\end{pmatrix}.\,
\end{equation}
since $[H, X_{\varepsilon_k - \varepsilon_\ell}]
=
(\varepsilon_k - \varepsilon_\ell)(H)\, X_{\varepsilon_k - \varepsilon_\ell}$, for $H\in \mathfrak{h}$.

\vspace{.7em}

\noindent \underline{Roots $\varepsilon_k + \varepsilon_\ell$ ($k \neq \ell$)}:  A root vector for $\varepsilon_k + \varepsilon_\ell$ is
\begin{equation}
X_{\varepsilon_k + \varepsilon_\ell}
\coloneqq
\begin{pmatrix}
0 & E_{k \ell} + E_{\ell k} \\
0 & 0
\end{pmatrix},
\end{equation}
since $[H, X_{\varepsilon_k + \varepsilon_\ell}]
= (\varepsilon_k + \varepsilon_\ell)(H)\, X_{\varepsilon_k + \varepsilon_\ell}$. 
The corresponding negative root vector is
\begin{equation}
X_{-(\varepsilon_k + \varepsilon_\ell)}
\coloneqq
\begin{pmatrix}
0 & 0 \\
E_{k \ell} + E_{\ell k} & 0
\end{pmatrix}.
\end{equation}
\vspace{.5em}

\noindent \underline{Roots $2\varepsilon_k$}: A root vector for $2\varepsilon_k$ is
\begin{equation}
X_{2\varepsilon_k}
\coloneqq
\begin{pmatrix}
0 & E_{kk} \\
0 & 0
\end{pmatrix},
\end{equation}
since $[H, X_{2\varepsilon_k}]
=
2\varepsilon_k(H)\, X_{2\varepsilon_k}$. 
The corresponding negative root vector is
\begin{equation}
X_{-2\varepsilon_k}
\coloneqq
\begin{pmatrix}
0 & 0 \\
E_{kk} & 0
\end{pmatrix}.
\end{equation}

With the explicit choice of Cartan subalgebra and root vectors above, we obtain the
root space decomposition
\begin{equation}\label{cdecomp}
\begin{aligned}
\mathfrak{sp}_{2n+2}(\mathbb C) &= \mathfrak{h}\oplus \bigoplus_{\alpha \in \Phi} (\mathfrak{sp}_{2n+2}(\mathbb C))_{\alpha}\\
&= \mathfrak{h} \oplus \bigoplus_{\alpha \in \Phi}  \mathbb{C}X_\alpha.
\end{aligned}
\end{equation}

The \emph{real symplectic algebra} $\mathfrak{sp}_{2n+2}(\mathbb{R})$ is the split--real form of $\mathfrak{sp}_{2n+2}(\mathbb{C})$, and is defined by 
\begin{equation}
\begin{aligned}
\mathfrak{sp}_{2n+2}(\mathbb{R})
&=\{X\in M_{2n+2}(\mathbb{R}) \mid X^{\top}J+JX=0\},
\end{aligned}
\end{equation}
with the same block form as for $\mathfrak{sp}_{2n+2}(\mathbb{C})$ in Eq.~\eqref{block}, but with blocks in $M_{n+1}(\mathbb{R})$ instead of $M_{n+1}(\mathbb{C})$.
Analogous to the decomposition in Eq.~\eqref{cdecomp} for $\mathfrak{sp}_{2n+2}(\mathbb{C})$,
we have the decomposition  
\begin{equation}\label{cdecompp}
\mathfrak{sp}_{2n+2}(\mathbb R)
=
\mathfrak{h}^{\mathbb R}
\oplus
\bigoplus_{\alpha \in \Phi} \mathbb{R}X_\alpha,
\end{equation}
 where $\mathfrak{h}^{\mathbb R}=\mathfrak{h}\cap\mathfrak{sp}_{2n+2}(\mathbb R)$.

\subsection{Indecomposable Representations of $\mathfrak{hw}_n$ via Symplectic Embedding}

Consider the closed subset of $\Phi$
\begin{equation}
T\coloneqq\{\varepsilon_1-\varepsilon_k,\ \varepsilon_1+\varepsilon_k \mid 2\le k\le n+1\}\ \cup\ \{2\varepsilon_1\},
\end{equation}
 and define a corresponding regular subalgebra of $\mathfrak{sp}_{2n+2}(\mathbb{C})$ by
\begin{equation}
\mathfrak{s}_{T}\coloneqq \bigoplus_{\alpha\in T} (\mathfrak{sp}_{2n+2}(\mathbb{C}))_\alpha
= \bigoplus_{\alpha\in T} \mathbb{C}X_\alpha.\end{equation}
By construction,
\begin{equation}\label{sT_spanb}
\mathfrak{s}_T
=
\operatorname{span}_{\mathbb{C}}\{X_{\varepsilon_1-\varepsilon_{k+1}}, \;X_{\varepsilon_1+\varepsilon_{k+1}}, \;X_{2\varepsilon_1}\}_{k=1}^n\subset \mathfrak{sp}_{2n+2}(\mathbb{C}).
\end{equation}
Since each of the generating elements above is in $\mathfrak{sp}_{2n+2}(\mathbb{R})$, we may analogously define a subalgebra of $\mathfrak{sp}_{2n+2}(\mathbb{R})$:
\begin{equation}\label{s_spanb}
\mathfrak{s}
\coloneqq
\operatorname{span}_{\mathbb{R}}\{X_{\varepsilon_1-\varepsilon_{k+1}}, \;X_{\varepsilon_1+\varepsilon_{k+1}}, \;X_{2\varepsilon_1}\}_{k=1}^n\subset \mathfrak{sp}_{2n+2}(\mathbb{R}).
\end{equation}
Moreover, it is clear that $\mathfrak{s}_T$ is the complexification of $\mathfrak{s}$:
\begin{equation}\label{complexification}
\mathfrak{s}^\mathbb{C}=\mathfrak{s}\otimes_{\mathbb{R}} \mathbb{C}=\mathfrak{s}_T.
\end{equation}

\begin{lemma}\label{lem.iso}
The regular subalgebra $\mathfrak{s}_{T}\subset\mathfrak{sp}_{2n+2}(\mathbb{C})$ is
isomorphic to the complexification of the Heisenberg--Weyl Lie algebra
$\mathfrak{hw}_{n}^{\mathbb{C}}$; and the subalgebra
$\mathfrak{s}\subset\mathfrak{sp}_{2n+2}(\mathbb{R})$ is isomorphic to the
Heisenberg--Weyl Lie algebra $\mathfrak{hw}_{n}$. That is,
\begin{equation}
\mathfrak{hw}_n^{\mathbb{C}} \cong \mathfrak{s}_T=\mathfrak{s}^{\mathbb{C}}
\quad \text{and} \quad
\mathfrak{hw}_n \cong \mathfrak{s}.
\end{equation}
\end{lemma}

\begin{proof}
We first establish $\mathfrak{hw}_n^{\mathbb{C}} \cong \mathfrak{s}_T$.
Define
\begin{equation}\label{xyzk_def}
x_k\coloneqq X_{\varepsilon_1-\varepsilon_{k+1}},\qquad
y_k\coloneqq X_{\varepsilon_1+\varepsilon_{k+1}},\qquad
z\coloneqq [x_1,y_1]=2X_{2\varepsilon_1}\in \mathbb{C}X_{2\varepsilon_1},
\end{equation}
for $1\le k\le n$.  By construction,
\begin{equation}\label{sT_span}
\mathfrak{s}_T
=
\operatorname{span}_{\mathbb{C}}\{x_1,\dots,x_n,\, y_1,\dots,y_n,\, z\}.
\end{equation}
Direct computation shows
\begin{equation}\label{bbcc}
\begin{aligned}
[x_k,y_\ell]&=\delta_{k\ell}\,z,\\
[x_k,x_\ell]&=[y_k,y_\ell]=[z,x_k]=[z,y_k]=0,
\end{aligned}
\qquad (1\le k,\, \ell\le n).
\end{equation}
Define $\Psi_{\mathbb{C}}:\mathfrak{hw}_{n}^{\mathbb C}\to\mathfrak{s}_T$ by
\begin{equation}
\Psi_{\mathbb{C}}(P_k)=x_k,\qquad \Psi_{\mathbb{C}}(Q_k)=y_k,\qquad \Psi_{\mathbb{C}}(Z)=z\qquad (1\le k\le n).
\end{equation}
Then Eq.~\eqref{commutation} and Eq.~\eqref{bbcc} imply that $\Psi_{\mathbb{C}}$ preserves Lie brackets, so $\Psi_{\mathbb{C}}$ is a Lie algebra homomorphism.
It is surjective by Eq.~\eqref{sT_span}, and by dimension considerations it follows that $\Psi_{\mathbb{C}}$ is bijective, hence a Lie algebra isomorphism.

We may also define $\Psi_{\mathbb{R}}: \mathfrak{hw}_n \rightarrow \mathfrak{s}$ by
\begin{equation}
\Psi_{\mathbb{R}}(P_k)=x_k,\qquad \Psi_{\mathbb{R}}(Q_k)=y_k,\qquad \Psi_{\mathbb{R}}(Z)=z\qquad (1\le k\le n),
\end{equation}
which, analogously to the complex case above, is  a Lie algebra isomorphism between $\mathfrak{hw}_n$ and $\mathfrak{s}$.
\end{proof}

\begin{lemma}\label{real_extensionb}
Let $\mathfrak g$ be a real semisimple Lie algebra and let
$\mathfrak g^{\mathbb C}$ be its complexification.
Let $\mathfrak s^{\mathbb C}\subset \mathfrak g^{\mathbb C}$ be a regular subalgebra,
and assume that $\mathfrak s^{\mathbb C}$ is the complexification of a real
subalgebra $\mathfrak s\subset \mathfrak g$.
Then, every finite--dimensional complex \emph{irreducible} representation of $\mathfrak g$ remains
indecomposable when restricted to $\mathfrak s$
if and only if every finite--dimensional complex \emph{irreducible} representation of $\mathfrak g^{\mathbb C}$ remains
indecomposable when restricted to $\mathfrak s^{\mathbb C}$.
\end{lemma}

\begin{proof}
Let $V$ be a finite-dimensional complex $\mathfrak{g}$-module. Then, $V$ extends uniquely to a complex-linear $\mathfrak{g}^\mathbb{C}$-module \cite[Proposition 4.6]{hall-lie}. Moreover, a complex subspace $V$ is $\mathfrak{g}$--invariant if and only if it is
$\mathfrak{g}^\mathbb{C}$-invariant. In particular, $V$ is $\mathfrak{g}$-irreducible if and only if it is $\mathfrak{g}^\mathbb{C}$-irreducible.

Similarly, $V$ is $\mathfrak{s}$-invariant if and only if it is invariant with respect to its complexification $\mathfrak{s}^\mathbb{C}$. Hence, $V$ is an indecomposable $\mathfrak{s}$-module if and only if it is indecomposable with respect to $\mathfrak{s}^\mathbb{C}$. The desired equivalence follows.
\end{proof}

\begin{remark}
\label{real_extension_remark}
Specializing Lemma~\ref{real_extensionb} to
$\mathfrak{s} \subset \mathfrak{sp}_{2n+2}(\mathbb{R})$ and
$\mathfrak{s}_T \subset \mathfrak{sp}_{2n+2}(\mathbb{C})$, and considering Eq.~\eqref{complexification}, the following
statements are equivalent:
\begin{itemize}
\item Every finite--dimensional irreducible complex representation of
$\mathfrak{sp}_{2n+2}(\mathbb{R})$ restricts to an
indecomposable representation of
$\mathfrak{s} \cong \mathfrak{hw}_n$.
\item Every finite--dimensional irreducible complex representation of 
$\mathfrak{sp}_{2n+2}(\mathbb{C})$ restricts to an
indecomposable representation of
$\mathfrak{s}_T \cong \mathfrak{hw}_n^{\mathbb{C}}$.
\end{itemize}
\end{remark}

Our final theorem establishes a large family of finite--dimensional indecomposable complex representations of the Heisenberg--Weyl Lie algebra $\mathfrak{hw}_n$.

\begin{theorem}
Let $\mathfrak{s}\subset \mathfrak{sp}_{2n+2}(\mathbb R)$ be the Heisenberg--Weyl
subalgebra defined in Eq.~\eqref{s_spanb}, with $\mathfrak{s}\cong\mathfrak{hw}_n$.
Then every finite--dimensional complex irreducible representation of
$\mathfrak{sp}_{2n+2}(\mathbb R)$ remains indecomposable upon
restriction to $\mathfrak{s}$.
\end{theorem}

\begin{proof}
By Theorem \ref{lwide} and Lemma \ref{real_extensionb}, or more precisely Remark \ref{real_extension_remark}, it suffices to prove $\Phi = [T\cup \, (-T)]$.
We decompose $\Phi$ as follows:
\begin{equation}
\begin{aligned}
\Phi
=\;&
\bigl\{\pm(\varepsilon_1-\varepsilon_k)\bigr\}_{k=2}^{n+1} \\[2pt]
&\;\cup\;
\bigl\{\pm(\varepsilon_1+\varepsilon_k)\bigr\}_{k=2}^{n+1} \\[2pt]
&\;\cup\;
\bigl\{\pm 2\varepsilon_1\bigr\} \\[2pt]
&\;\cup\;
\bigl\{\pm(\varepsilon_k-\varepsilon_\ell)\bigr\}_{2\le k<\ell\le n+1} \\[2pt]
&\;\cup\;
\bigl\{\pm(\varepsilon_k+\varepsilon_\ell)\bigr\}_{2\le k<\ell\le n+1} \\[2pt]
&\;\cup\;
\bigl\{\pm 2\varepsilon_k\bigr\}_{k=2}^{n+1}.
\end{aligned}
\end{equation}
The first three families of roots, namely
\begin{equation}\bigl\{\pm(\varepsilon_1-\varepsilon_k)\bigr\}_{k=2}^{n+1},\;\bigl\{\pm(\varepsilon_1+\varepsilon_k)\bigr\}_{k=2}^{n+1}, \;\text{and}\; \bigl\{\pm 2\varepsilon_1\bigr\},\end{equation}
consist   of elements of $\pm T$, and hence  are elements of $ [T\cup \, (-T)]$.

The remaining roots can be written as sums of elements of $T\cup(-T)$, which thus belong to $[T\cup(-T)]$ by closure.
To illustrate these inclusions, respectively, for \(2\le k<\ell \le n+1\),
\begin{equation}
\begin{aligned}
\pm(\varepsilon_k - \varepsilon_\ell) &=\;\pm( \varepsilon_{1} -\varepsilon_{\ell} -( \varepsilon_{1} -\varepsilon_{k}))  \in  [T\cup (-T)],\\
\pm(\varepsilon_k + \varepsilon_\ell) &=\;\pm( \varepsilon_{1} +\varepsilon_{\ell} -( \varepsilon_{1} -\varepsilon_{k}))  \in  [T\cup (-T)],\\
\pm 2\varepsilon_k &=\; \pm(\varepsilon_{1}+ \varepsilon_{k} - (\varepsilon_{1} -\varepsilon_k))\in  [T\cup (-T)].
\end{aligned}
\end{equation}

Thus $\Phi \subseteq [T\cup \,(-T)]$. Since the reverse inclusion is automatic,
we have $\Phi = [T\cup \,(-T)]$, as required. 
\end{proof}

\begin{remark}[Tensor product decomposition under restriction]
Tensor product decompositions of finite--dimensional complex irreducible
representations of $\mathfrak{sp}_{2n+2}(\mathbb R)$ induce corresponding decompositions of
their restrictions to $\mathfrak{hw}_n$.
In particular, if a tensor product of complex irreducible representations of
$\mathfrak{sp}_{2n+2}(\mathbb R)$ decomposes as a direct sum of
irreducible representations of $\mathfrak{sp}_{2n+2}(\mathbb R)$, then the restricted representations of $\mathfrak{hw}_n$ decompose as a direct sum of the corresponding
indecomposable representations of $\mathfrak{hw}_n$.
Thus, combinatorial and representation--theoretic rules governing tensor product decompositions for 
$\mathfrak{sp}_{2n+2}(\mathbb R)$ apply directly to the decomposition of the associated representations of $\mathfrak{hw}_n$.
\end{remark}

\begin{remark}[Why the finite--dimensional representations are nonunitary]
Finite--dimensional unitary representations are completely reducible; in
particular, a finite--dimensional unitary representation decomposes as a direct
sum of irreducible unitary representations.
By \cite[Theorem~1.59]{folland}, every finite--dimensional irreducible unitary
representation of the real Heisenberg--Weyl Lie algebra $\mathfrak{hw}_n$ is
one--dimensional.
Therefore any finite--dimensional unitary $\mathfrak{hw}_n$--representation is a direct sum
of one--dimensional representations.
Since the $\mathfrak{hw}_n$--representations constructed in this section have dimension
greater than one and are indecomposable, they cannot be unitary.
\end{remark}

\section{Conclusion}

We have presented in this paper two distinct results on the Heisenberg-Weyl group and its algebra.   We have explicitly shown that the tensor product of two Schr\"odinger representations 
 with central characters $\lambda$ and 
$\mu$ is unitarily equivalent, when $\lambda+\mu\neq 0$, to the 
Schr\"odinger representation with central character $\lambda+\mu$ 
with infinite multiplicity.

Despite the close connection between the Heisenberg--Weyl algebra and the quantum harmonic oscillator, we observe (see Appendix) that the lowest weight state of any copy of the irrep with central character $\lambda+\mu$ in the decomposition is not a ground state for the sum of the two oscillators (except for the lowest weight state given by the product of the two individual lowest weight states).  Thus, the action of the raising operator on a lowest weight state cannot result in an eigenfunction of the harmonic oscillator either.  

Next, we constructed a natural family of finite--dimensional indecomposable representations of $\mathfrak{hw}_n$ via symplectic embedding, which are nonunitary. Specifically, we identified a subalgebra of $\mathfrak{sp}_{2n+2}(\mathbb{R})$ isomorphic to   $\mathfrak{hw}_n$  and showed that the restriction of every finite--dimensional irreducible representation of $\mathfrak{sp}_{2n+2}(\mathbb{R})$ remained indecomposable when restricted the  Heisenberg--Weyl subalgebra. This provides a large family of nonunitary indecomposable representations of  $\mathfrak{hw}_n$.

\section*{Acknowledgments}

The work of HdG is funded by NSERC of Canada.

\setcounter{section}{0}
\renewcommand\thesection{\Alph{section}}
\section{Appendix: Constructing lowest weight states}
\label{sec:appendix}
\smallskip 

In this appendix we give an explicit construction of lowest weight states using Hermite polynomials as the well-known basis functions used in quantum mechanics. We assume throughout $\lambda> 0$ and $\mu> 0$.

\subsection{Single oscillator}

Acting on (complex) functions of the real variable $u$, the derived 
Schr\"odinger representation of the Heisenberg--Weyl Lie algebra 
$\mathfrak{hw}_1$ is given by
\begin{align}
    d\pi_\lambda(Q)\psi^{\lambda}(u)&=iq\psi^{\lambda}(u)\, , \nonumber \\
    d\pi_\lambda(P)\psi^{\lambda}(u)&=\lambda \frac{d}{dq}\psi^{\lambda}(u)\, ,
    \label{eq:standardrepresentation} \\
    d\pi_\lambda(Z)\psi^{\lambda}(u)&=i\lambda \psi^{\lambda}(u)\nonumber \, ,
\end{align} 
where 
\begin{align}
    \int_{-\infty}^\infty du \vert \psi^\lambda(u)\vert^2 < \infty\, .
\end{align}
Thus the central character $\lambda$ appears as the eigenvalue of the central element $Z$, and one easily verifies (obviously)
\begin{align}
    [d\pi_\lambda(Q),d\pi_\lambda(P)]\psi^\lambda(u)=i\lambda \psi^\lambda(u)
\end{align}

As set of (normalized) basis functions $\{\psi_n^\lambda(u), n=0,1,\ldots,\infty\}$  for the carrier space of the group representation obtained by exponentiating the action of Eq.~(\ref{eq:standardrepresentation}) is given by the product of a Gaussian factor and Hermite polynomials:
\begin{align}
    \psi_n^\lambda(u)=(2^n n!)^{-1/2} (\pi\lambda)^{-1/4} e^{-u^2/2\lambda}H_n(u/\sqrt{\lambda})\, ,\qquad \lambda>0\, .
\end{align}
In particular, we have the harmonic oscillator Hamiltonian operator
\begin{equation}
\begin{aligned}
    \tilde{H}\psi^\lambda_n(u):= \frac12 \left((d\pi_\lambda(P))^2 +(d\pi_\lambda Q)^2\right)\psi^\lambda_n(u)
    &\mapsto \frac12 \left(\lambda^2\frac{d^2}{du^2}-u^2\right)\psi^\lambda_n(u)\\
    &= -\lambda (n+\textstyle\frac12)\psi^\lambda_n(u)\, ,\quad n\in \mathbf{Z}^+ \, .
    \label{eq:hostates}
\end{aligned}
\end{equation}
The sign reversal compared to ``standard'' quantum mechanics expression is due to our placement of the factor $i$ in the representation $d\pi_\lambda$.
Comparing with the standard harmonic oscillator Hamiltonian $H$ of quantum mechanics  we find that $-\tilde H=H$ and that 
$\lambda=\hbar$, \textit{i.e.} the central character is related to the ``Planck constant'' of the oscillator.  

Eq.(\ref{eq:action}) can be used to define the matrix elements $X_{mn}$ of an operator by 
\begin{align}
    X_{mn}=\int_{-\infty}^{\infty}du\,
    (d\pi_{\lambda}(X)\psi_n^\lambda(u)) \overline{\psi_{m}^{\lambda}(u)}
\end{align} and rapidly leads to 
\begin{equation}
\begin{aligned}
    Q_{mn}=i \int_{-\infty}^\infty du \psi^{\lambda}_n(u) u \psi^{\lambda}_m(u)\, ,\\
    P_{mn}=\lambda \int_{-\infty}^\infty du \psi^{\lambda}_n(u) \left(\frac{d}{du}\right)\psi^{\lambda}_m(u)
\end{aligned}
\end{equation}
    and the (infinite-dimensional) representation: 
    \begin{equation}
    \begin{aligned}
        Q &\mapsto i \sqrt{\frac{\lambda}{2}}
        \left(\begin{array}{ccccc}
        0&1&0&0&\ldots \\
        1&0&\sqrt{2}&0&\ldots \\
        0&\sqrt{2}&0&\sqrt{3}&\ldots\\
        0&0&\sqrt{3}&0 &\ldots \\
        \vdots & \vdots & \vdots &\vdots &
        \end{array}\right)\, ,\\
        P&\mapsto \sqrt{\frac{\lambda}{2}}
        \left(\begin{array}{ccccc}
        0&1&0&0&\ldots \\
        -1&0&\sqrt{2}&0&\ldots \\
        0&-\sqrt{2}&0&\sqrt{3}&\ldots\\
        0&0&-\sqrt{3}&0 &\ldots \\
        \vdots & \vdots & \vdots &\vdots &
        \end{array}\right)
    \end{aligned}
    \end{equation}
(Note again the placement of the $i$'s, which follow from the definitions of Eq.~(\ref{eq:horiginal})).
As always, one can construct a raising and a lowering operators: 
\begin{equation}
\begin{aligned}
    \hat{a}^\dagger&= \frac{1}{\sqrt{2\lambda}}\left(-i  Q- P\right)\, ,\\
    \hat{a}&= \frac{1}{\sqrt{2\lambda}}\left(-i  Q + P\right)
\end{aligned}
\end{equation}
so that 
\begin{align}
    \hat a\psi^\lambda_n(u)&=\sqrt{n}\psi^\lambda_{n-1}(u)\, ,\label{eq:alower}\\
    \hat a^\dagger\psi^\lambda_n(u)&=\sqrt{n+1}\psi^\lambda_{n+1}(u)\nonumber \, .
\end{align}
The ground state, or equivalently the lowest weight state, is just $\psi_0^{\lambda}(u)$.

\subsection{Constructing lowest weight states in the decomposition}

The action of $\hat a$ and $\hat a^\dagger$ 
does not depend on $\lambda$.  Our job in 
this section is to find linear combination (not necessarily normalized w/r to the inner product)
of 
$\psi^{\lambda}_n(u)\psi^{\mu}_m(v)$ 
so they are annihilated by 
$\hat A=\hat a_1\otimes \mathbf{1}+\mathbf{1}\otimes \hat a_2$ where, in the obvious manner
\begin{equation}
\begin{aligned}
    \hat a_1&=\frac{1}{\sqrt{2\lambda}}(-iQ_1+P_1)\, ,  \\
\hat a_2&= 
    \frac{1}{\sqrt{2\mu}}(-iQ_2+P_2)\, .
\end{aligned}
\end{equation}

The action of each factor in $\hat A$ is linear so clearly 
\begin{align}
    \hat A\psi_0^{\lambda}(u)\psi_0^{\mu}(v)=0
\end{align}
and $\Psi_{0}(u,v):=\psi_0^{\lambda}(u)\psi_0^{\mu}(v)$ is our first lowest state. 
We next consider 
\begin{equation}
\begin{aligned}
    \hat A \left(\psi_0^{\lambda}(u)\psi_1^{\mu}(v)-\psi_1^{\lambda}(u)\psi_0^{\mu}(v)\right)&=
    \hat a_1\psi_0^{\lambda}(u)\psi_1^{\mu}(v)+\hat a_2\psi_0^{\lambda}(u)\psi_1^{\mu}(v) \\
    &\quad -\hat a_1\psi_1^{\lambda}(u)\psi_0^{\mu}(v)-\hat a_2\psi_1^{\lambda}(u)\psi_0^{\mu}(v)\, , \\
    &=0+\psi_0^{\lambda}(u)\psi_0(v)-
    \psi_0^{\lambda}(u)\psi_0(v)-0=0\, .
\end{aligned}
\end{equation}
so $\Psi_1(u,v)=\left(\psi_0^{\lambda}(u)\psi_1^{\mu}(v)-\psi_1^{\lambda}(u)\psi_0^{\mu}(v)\right)$ is our second lowest weight state.

In fact, the action of (say) $\hat a_1\psi_k^{\lambda}(u)$ is the same as the action 
\begin{align}
  a_1\psi_k^{\lambda}(u) \mapsto  \frac{\partial}{\partial \xi_1} \frac{\xi_1^k}{\sqrt{k!}}
\end{align}
where the standard polynomial inner product $\langle \xi_1^p,\xi_1^q\rangle= \delta_{pq}q!$, and we identify
\begin{align}
    \frac{\xi_1^p}{\sqrt{p!}}\leftrightarrow 
    \psi^\lambda_p(u)\, .
\end{align}
The problem of finding the lowest weight state is just that of finding polynomials in the dummies $\xi_1,\xi_2$ killed by $\frac{\partial}{\partial \xi_1}+\frac{\partial}{\partial \xi_2}$.  By inspection, these are polynomials of the form $(\xi_1-\xi_2)^k$ where
\begin{align}
    \xi_1^p \xi_2^q\mapsto \sqrt{p!q!}
    \psi_p^\lambda(u)\psi_q^{\mu}(v)
\end{align}
The first few are given explicitly in Table \ref{tab:lowestwxi}. 
This clearly shows there are countably infinitely many lowest weights.  They are all eigenstates of the central element $Z$ with eigenvalue $\lambda+\mu$, and a full set of states spanning each copy of this irrep can be obtained by repeated application of the operator $\hat A^\dagger=\hat a_1^\dagger\otimes \mathbf{1}+\mathbf{1}\otimes \hat a_2^\dagger$ on a lowest weight.

The form of the lowest weight states is also interesting.  
Explicit expressions for the first few functions are given in 
Table \ref{tab:lwH}, where $\sim$ indicates we have left out some constant factors and $H_n$ is the $n$'th Hermite polynomial.

It is clear from Table \ref{tab:lowestwxi} that, for $\lambda\ne \mu$, the lowest weight states are \textit{not} eigenstates of the two-particle oscillator $H_1+H_2$, with $H$ defined in Eq.~(\ref{eq:hostates}).  
For $\lambda=\mu$, the lowest weight states can be expressed as a product of harmonic oscillator states $\psi_0(\tfrac{u+v}{\sqrt{2\lambda}})
\psi_n(\tfrac{u-v}{\sqrt{2\lambda}})$.  This case is well known as the circular oscillator, and one can obtain various sets of basis states by using polar coordinates, although the interpretation of these as states in irreducible representations of $\mathfrak{hw}$ is not common in physics.

 \bigskip 

\begin{table}[ht!]
\centering
\renewcommand{\arraystretch}{1.6}
\begin{tabular}{c c l}
\hline
$k$ & Polynomial & Function \\
\hline

$0$
&
$\displaystyle 1$
&
$\displaystyle
\psi_0^{\lambda}(u)\psi_0^{\mu}(v)
$
\\

$1$
&
$\displaystyle \xi_1-\xi_2$
&
$\displaystyle
\psi_1^{\lambda}(u)\psi_0^{\mu}(v)
-
\psi_0^{\lambda}(u)\psi_1^{\mu}(v)
$
\\

$2$
&
$\displaystyle \xi_1^2-2\xi_1\xi_2+\xi_2^2$
&
$\displaystyle
\sqrt{2}\,\psi_2^\lambda(u)\psi_0^{\mu}(v)
-2\,\psi_1^\lambda(u)\psi_1^{\mu}(v)
+\sqrt{2}\,\psi_0^\lambda(u)\psi_2^{\mu}(v)
$
\\

\hline
\end{tabular}
\caption{Some lowest weight polynomials and corresponding tensor product basis functions.}
\label{tab:lowestwxi}
\end{table}

 \begin{table}[ht!]
\centering
\renewcommand{\arraystretch}{1.6}
\begin{tabular}{c l l}
\hline
$k$ & Explicit expression & Simplified form \\
\hline

$0$
&
$\displaystyle
\frac{e^{-\frac{u^2}{2\lambda}-\frac{v^2}{2\mu}}}
{\sqrt{\pi}\,\lambda^{1/4}\mu^{1/4}}
$
&
$\displaystyle
\sim
e^{-\frac{u^2}{2\lambda}-\frac{v^2}{2\mu}}
\,H_0\!\left(
\frac{\sqrt{\lambda}\,v-\sqrt{\mu}\,u}{\sqrt{2\lambda\mu}}
\right)
$
\\[2pt]
$1$
&
$\displaystyle
\frac{
e^{-\frac{u^2}{2\lambda}-\frac{v^2}{2\mu}}
\left(\sqrt{\lambda}\,v-\sqrt{\mu}\,u\right)
}
{\sqrt{\pi}\,\lambda^{3/4}\mu^{3/4}}
$
&
$\displaystyle
\sim
e^{-\frac{u^2}{2\lambda}-\frac{v^2}{2\mu}}
\,H_1\!\left(
\frac{\sqrt{\lambda}\,v-\sqrt{\mu}\,u}{\sqrt{2\lambda\mu}}
\right)
$
\\[2pt]
$2$
&
$\displaystyle
\frac{
2\,e^{-\frac{u^2}{2\lambda}-\frac{v^2}{2\mu}}
\Bigl(
\lambda v^2
-2\sqrt{\lambda\mu}\,uv
+\mu u^2
-\lambda\mu
\Bigr)
}
{\sqrt{\pi}\,\lambda^{5/4}\mu^{5/4}}
$
&
$\displaystyle
\sim
e^{-\frac{u^2}{2\lambda}-\frac{v^2}{2\mu}}
\,H_2\!\left(
\frac{\sqrt{\lambda}\,v-\sqrt{\mu}\,u}{\sqrt{2\lambda\mu}}
\right)
$
\\
\hline
\end{tabular}
\caption{Lowest weight states as explicit functions.}
\label{tab:lwH}
\end{table}


\begin{thebibliography}{99}


\bibitem{ali-englis-rmp}
S.~T.~Ali and M.~Engli\v{s},
``Quantization methods: a guide for physicists and analysts'',
\textit{Rev.~Math.~Phys.} \textbf{17} (2005), 391--490.

\bibitem{Almheiri}
A.~Almheiri, A.~Goel and X.Y.~Hu, . 
``Quantum gravity of the Heisenberg algebra'', Journal of High Energy Physics, \textbf{8} (2024), pp.1-40.



\bibitem{arai}
A.~Arai,
\textit{Inequivalent Representations of Canonical Commutation and
Anti-Commutation Relations},
Springer, Singapore, 2020.

\bibitem{Born}
M.~Born,
\textit{Physics in my generation.  (A selection of papers)},
Pergamon Press, London, 1960



\bibitem{corwin-greenleaf}
L.~J.~Corwin and F.~P.~Greenleaf,
\textit{Representations of Nilpotent Lie Groups and Their Applications},
Cambridge Studies in Advanced Mathematics, vol.~18,
Cambridge University Press, Cambridge, 1990.


\bibitem{Cushman}
R.~Cushman, ``A momentum map for the Heisenberg group'', Symmetry \textbf{16} (2024) 1054.


\bibitem{DorfmeisterDorfmeister}
G.~Dorfmeister and J.~Dorfmeister, ``Classification of certain pairs of operators (P, Q) 
satisfying [P, Q]=-i Id'', Journal of Functional Analysis, 57 (1983) vol. 3, pp.301-328.

\bibitem{DouglasRepka2025}
A.~Douglas and J.~Repka,
``Narrow and wide regular subalgebras of semisimple Lie algebras'',
\textit{J.~Algebra} \textbf{664} (2025), 348--361.

\bibitem{folland}
G.~B.~Folland,
\textit{Harmonic Analysis in Phase Space},
Annals of Mathematics Studies, vol.~122,
Princeton University Press, Princeton, NJ, 1989.

\bibitem{folland-aha}
G.~B.~Folland,
\textit{A Course in Abstract Harmonic Analysis},
2nd ed.,
Studies in Advanced Mathematics,
CRC Press, Boca Raton, FL, 2016. 


\bibitem{folland2025}
G.~B.~Folland,
\textit{The Heisenberg Group: A Survey},
Colloquium Publications, vol.~68,
American Mathematical Society, Providence, RI, 2025.

\bibitem{Frahm}
J.~Frahm, C.~Weiske and G.~Zhang,  2023. 
``Heisenberg parabolically induced representations of Hermitian Lie groups, Part I: Intertwining operators and Weyl transform'', Advances in Mathematics,
\textbf{422} (2023), p.109001.


\bibitem{Garrison-Wong}
J.~C.~Garrison and J.~Wong. ``Canonically conjugate pairs, uncertainty relations, and phase operators'', Journal of Mathematical Physics 
vol. 11 (1970) no 8, 2242-2249.


\bibitem{Glauber}
R.~J.~Glauber, ``Coherent and Incoherent States of the Radiation Field'', Phys. Rev. \textbf{131} (1963) 2766


\bibitem{degosson}
M.~de~Gosson,
\textit{Symplectic Geometry and Quantum Mechanics},
Operator Theory: Advances and Applications, vol.~166,
Birkh\"auser, Basel, 2006.

\bibitem{hall}
B.~C.~Hall,
\textit{Quantum Theory for Mathematicians},
Graduate Texts in Mathematics, vol.~267,
Springer, New York, 2013.

\bibitem{hall-lie}
B.~C.~Hall,
\textit{Lie Groups, Lie Algebras, and Representations: An Elementary Introduction},
Graduate Texts in Mathematics, vol.~222,
Springer, New York, 2015.



\bibitem{howe}
R.~Howe,
``The role of the Heisenberg group in harmonic analysis'',
\textit{Bull.~Amer.~Math.~Soc.} (N.S.) \textbf{3} (1980), no.~2, 821--843.

\bibitem{KnappLieGroups}
A.~W.~Knapp,
\textit{Lie Groups Beyond an Introduction},
2nd ed., Progress in Mathematics, vol.~140,
Birkh\"auser, Boston, 2002.


\bibitem{mackey}
G.~W.~Mackey,
\textit{Induced Representations of Groups and Quantum Mechanics},
W.~A.~Benjamin, Inc., New York--Amsterdam, 1968.

\bibitem{panyushev}
D.~I.~Panyushev,
``Wide subalgebras of semisimple Lie algebras,''
\textit{Algebr.~Represent.~Theory} \textbf{17} (2014), no.~3, 931--944.

\bibitem{Perelomov}
A.~Perelomov, ``Generalized coherent states and some of their applications'', Soviet Physics Uspekhi, \textbf{20} (1977) 703-720.


\bibitem{Schleich}
W.~P.~Schleich, \textit{Quantum optics in phase space}. John Wiley \& Sons, 2015.

\bibitem{stone1932} M.~H.~Stone,  ``Linear transformations in Hilbert space and their applications to analysis (Vol. 15)'', 
\textit{American Mathematical Soc.} (1932)

\bibitem{Sudarshan}
E.~C.~G. Sudarshan, ``Equivalence of Semiclassical and Quantum Mechanical Descriptions of Statistical Light Beams'', Phys. Rev. Lett. \textbf{10} (1963) 277


\bibitem{summers-stone-vn}
S.~J.~Summers,
``On the Stone--von Neumann uniqueness theorem and its ramifications,''
in \textit{John von Neumann and the Foundations of Quantum Physics},
Springer,  2001, pp.~135--152.


\bibitem{thangavelu}
S.~Thangavelu,
\textit{Harmonic Analysis on the Heisenberg Group},
Progress in Mathematics, vol.~159,
Birkh\"auser, Boston, 1998.

\bibitem{VilenkinKlimyk}
N.~Ja.~Vilenkin and A.~U.~Klimyk,
\textit{Representation of Lie Groups and Special Functions, Volume~2:
Class~I Representations, Special Functions, and Integral Transforms},
Kluwer Academic Publishers, Dordrecht, 1993.


\bibitem{weyl}
H.~Weyl,
``Quantenmechanik und Gruppentheorie'',
Zeitschrift für Physik \textbf{46} (1927), 1--46.

\bibitem{Wolf}
K.~B. Wolf, \textit{The Heisenberg-Weyl ring in quantum mechanics}, in Group theory and its applications Vol. III ed EM Loebl (1986)

\end{thebibliography}
\end{document}